\documentclass[11pt]{amsart}
\usepackage{graphicx}
\usepackage{epstopdf}
\usepackage{url}
\usepackage{latexsym}
\usepackage{amsfonts,amsmath,amssymb}
\newtheorem{theorem}{Theorem}[section]
\newtheorem{proposition}[theorem]{Proposition}
\newtheorem{lemma}[theorem]{Lemma}
\newtheorem{definition}[theorem]{Definition}

\newtheorem{example}[theorem]{Example}

\date{\today}

\begin{document}

\title[Nonrational generating functions]{Most principal permutation classes have nonrational generating functions}

\author{Mikl\'os B\'ona}
\address{Department of Mathematics, University of Florida, $358$ Little Hall, PO Box $118105$,
Gainesville, FL, $32611-8105$ (USA)}
\email{bona@ufl.edu}

\begin{abstract} We prove that for any fixed $n$, and for most permutation patterns $q$, the number
$\textup{Av}_{n,\ell}(q)$ of $q$-avoiding permutations of length $n$ that consist 
of $\ell$ skew blocks is a monotone decreasing function of $\ell$. We then show that this implies that for most patterns $q$,  the generating function $\sum_{n\geq 0} \textup{Av}_n(q)z^n$ of the sequence $\textup{Av}_n(q)$ of the numbers
of $q$-avoiding permutations is not rational. Placing our results in a broader context, we show that for rational 
power series $F(z)$ and $G(z)$ with nonnegative real coefficients, the relation $F(z)=1/(1-G(z))$ is supercritical, 
while for most permutation patterns $q$, the corresponding relation is not supercritical. 
\end{abstract}

\maketitle

\section{Introduction}
 We say that a permutation $p$ {\em contains} the pattern $q=q_1q_2\cdots q_k$ 
if there is a $k$-element set of indices $i_1<i_2< \cdots <i_k$ so that $p_{i_r} < p_{i_s} $ if and only
if $q_r<q_s$.  If $p$ does not contain $q$, then we say that $p$ {\em avoids} $q$. For example, $p=3752416$ contains
$q=2413$, as the first, second, fourth, and seventh entries of $p$ form the subsequence 3726, which is order-isomorphic
to $q=2413$.  A recent survey on permutation 
patterns can be found in \cite{vatter} and a book on the subject is \cite{combperm}. Let  $\textup{Av}_n(q)$ be 
the number of permutations of length $n$ that avoid the pattern $q$. In general, it is very difficult to compute, or even describe,
the numbers  $\textup{Av}_n(q)$, or their sequence as $n$ goes to infinity. As far as the generating function
$A_q(z)=\sum_{n\geq 0} \textup{Av}_n(q)z^n$ goes, there are known examples when it is algebraic, (when
$q$ is of length three, or when $q=1342$), and known examples when it is not algebraic (when $q$ is the monotone
pattern $12\cdots k$, where $k$ is an even integer that is at least four). The question
whether $A_q(z)$ is always differentiably finite was raised in 1996 by John Noonan and Doron Zeilberger, and is still open. 
Garrabrant and Pak \cite{garrabrant} \cite{garrabrant1} have recently showed that if $S$ is a finite set of permutation patterns, and $A_S(z)$ is the generating
function enumerating permutations of each length $n$ that avoid {\em all} elements of $S$, then $A_S(z)$ is not 
always differentiably finite. However, no such result is known in the case that we study in this paper, that  is, when 
$S$ consists of just one pattern $q$. 
See Chapter 6 of \cite{stanleyec2} for an introduction to the theory of differentiably finite generating functions and their
importance. 

In this paper, we prove that for patterns $q=q_1q_2\cdots q_k$, where $k>2$ and $\{q_1,q_k\}\neq \{1,k\}$, the
generating function $A_q(z)$ is {\em never rational}, and this holds even for a few patterns for which $\{q_1,q_k\}=\{1,k\}$.
 It is plausible to think that our result holds for the less than   $1/[k(k-1)]$ of 
patterns of length $k$ for which we cannot prove it. On the other hand, the statement obviously fails for the
pattern $q=12$, since for that $q$, we trivially have that  $\textup{Av}_n(q)=1$ for all $n$, so 
$A_q(z)=1/(1-z)$. The set of permutations of any length that avoid a given pattern $q$ is often called a {\em principal
permutation class}, explaining the title of this paper. As rational functions are differentiably finite, this paper excludes a
small subset of differentiably finite power series from the set of possible generating functions of principal permutation classes. 

In proving the result described in the preceding paragraph, our main tool will be a theorem that is interesting on its own right. 
We say that a permutation $p$ is {\em skew indecomposable} if it is not possible to cut $p$ into two parts so that
each entry before the cut is larger than each entry after the cut. For instance, $p=3142$ is skew indecomposable,
but $r=346512$ is not as we can cut it into two parts by cutting between entries 5 and 1, to obtain $3465|12$. 

If $p$ is not skew indecomposable, then there is a unique way to cut $p$ into  nonempty skew indecomposable strings
$s_1,s_2,\cdots ,s_\ell$ of consecutive entries so that each entry of $s_i$ is larger than each entry of $s_j$ if $i<j$. We call these
strings $s_i$ the {\em skew blocks} of $p$. For instance, $p=67|435|2|1$ has four skew blocks, while skew indecomposable
permutations have one skew block. 

The number of skew blocks of a permutation is of central importance for this paper. For permutations with no restriction, 
it is easy to prove that almost all permutations of length $n$ are skew indecomposable. In this paper, we consider a similar
question for pattern avoiding permutations. We prove that if $q$ is a skew indecomposable pattern, and $n$ is any fixed positive integer, then the number $\textup{Av}_{n,\ell}(q)$ of $q$-avoiding permutations of length $n$ that consist 
of $\ell$ skew blocks is a monotone decreasing function of $\ell$. That is, as the number $\ell$ of skew blocks increases, the number of $q$-avoiding permutations with $\ell$ skew blocks decreases. We will only need a special case of these inequalities (the one relating to $\ell=1$ and $\ell=2$) to prove our main 
result in Section \ref{secnonrat}.

In Section \ref{broader} we place our results into a broader context by discussing them from the perspective of 
supercritical relations, which we introduce in Definition \ref{supcrit}. We show that our results imply that on the one hand, rational generating functions lead to 
supercritical relations (Theorem \ref{general}), while for most principal permutation classes, the corresponding relations defined by $A_q(z)$  are not 
supercritical (Theorem \ref{notsup}), proving that $A_q(z)$ is not rational. 

Theorem \ref{general} can be used to show that some other combinatorial generating functions are not rational. 
Present author has recently \cite{survey} used this technique to prove that for all $t$, the generating function counting 
$t$-stack sortable permutations of length $n$ is not rational.

\section{Preliminaries} \label{secprelim}

The following proposition shows that in order to prove our monotonicity result announced in the introduction, 
it suffices to prove the relevant inequality for $\ell=1$. This proposition does not hold for patterns that are not
skew indecomposable. Recall that $\textup{Av}_{n,\ell}(q)$ denotes the number of $q$-avoiding permutations of length
$n$ that consist of $\ell$ skew blocks.
\begin{proposition} \label{justone}  Let $q$ be any skew indecomposable pattern. If, for all positive integers $n$, the inequality
\begin{equation} \label{compare} \textup{Av}_{n,2}(q) \leq  \textup{Av}_{n,1}(q) \end{equation} 
holds, then for all positive integers $n$, and all positive integers
$\ell$, the inequality 
\[\textup{Av}_{n,\ell+1}(q) \leq  \textup{Av}_{n,\ell}(q) \] holds.
\end{proposition}

\begin{proof} Let $A_{\ell,q}(z)=\sum_{n\geq 1} \textup{Av}_{n,\ell}(q) z^n$ be the ordinary generating function of the 
sequence of the numbers  $\textup{Av}_{n,\ell}(q)$. As $q$ is skew indecomposable, a permutation $p$ with $\ell$
skew blocks is $q$-avoiding if and only if each of its skew blocks is $q$-avoiding. This implies that
$A_{\ell,q}(z)= A_{1,q}(z)^\ell$, so, for all $\ell\geq 2$, we have the equalities
\begin{equation} \label{powerrule} A_{\ell,q}(z)= A_{\ell-1,q}(z) \cdot A_{1,q}(z),\end{equation} and
\begin{equation} \label{powerrule2}  A_{\ell+1,q}(z)= A_{\ell-1,q}(z) \cdot A_{2,q}(z).\end{equation}
As the coefficient of each term in $A_{1,q}(z)$ is at least as large as the corresponding coefficient of $A_{2,q}(z)$, and the 
coefficients of $A_{\ell-1,q}(z)$,   $A_{1,q}(z)$, and  $A_{2,q}(z)$ are all nonnegative, it follows from the way in which
the product of power series is computed that the  coefficient of each term in $A_{\ell,q}(z)$ is at least as large as the corresponding coefficient of $A_{\ell+1,q}(z)$. This proves our claim.  
\end{proof}

We will also need the following simple fact. If $q=q_1q_2\cdots q_k$ is a pattern, let $q^r$ denote its reverse $q_kq_{k-1}\cdots q_1$, and let 
$q^c$ denote its complement, the pattern $(k+1-q_1)(k+1-q_2)\cdots (k+1-q_k)$.  For instance, if $q=25143$, then $q^r=34152$, and $q^{c}=41523$. Recall that $\textup{Av}_n(q)$ denotes the number of 
permutations of length $n$ that avoid $q$. It is then obvious that for all patterns $q$, the equalities
\begin{equation} \label{trivial} \textup{Av}_n(q)=\textup{Av}_n(q^r) =\textup{Av}_n(q^c) \end{equation}
hold. These equalities, and similar others, will be useful for  us because of the following fact. 

\begin{proposition} \label{equality} Let $q$ and $q'$ be two skew indecomposable patterns so that the equality
\begin{equation} \label{alln} \textup{Av}_n(q) = \textup{Av}_n(q') \end{equation} holds for all $n\geq 1$. Then for 
all positive integers $n$, and for all positive integers $\ell \leq n$, the equality
\begin{equation} \label{allk}  \textup{Av}_{n,\ell}(q) =  \textup{Av}_{n,\ell}(q') \end{equation} holds.
\end{proposition}

In other words, if two skew indecomposable patterns are avoided by the same number or permutations of length $n$ for
all $n$, (in this case they are called {\em Wilf-equivalent}), then they are avoided by the same number of permutations
of  length $n$ that have $\ell$ skew blocks. 

\begin{proof} Recall that $A_q(z)=\sum_{n\geq 0} \textup{Av}_n(q) z^n$. Then
 \[A_q(z)=\sum_{\ell \geq 0}  A_{\ell,q}(z) = \sum_{\ell \geq 0} \left (  A_{1,q}  (z) \right )^{\ell}= \frac{1}{1-A_{1,q}(z)},\]
where the second equality follows from \eqref{powerrule} since $q$ is skew indecomposable. 

Therefore, \begin{equation} \label{indec} A_{1,q}(z) = 1-\frac{1}{A_q(z)},\end{equation}
and similarly, 
\[A_{1,q'}(z) = 1 - \frac{1}{A_{q'}(z)}.\]

Therefore, our conditions imply that $A_{1,q}(z)=A_{1,q'}(z)$, and therefore, for all $\ell$, the equalities
\[A_{\ell,q}(z) = (A_{1,q}(z))^\ell =  (A_{1,q'}(z))^\ell=A_{\ell,q'}(z)\] hold. Equating coefficients of $z^n$ completes our
proof.
\end{proof}

\section{The pattern 132} \label{section132}

The pattern 132 will be of particular importance to us because it enables us to illustrate a method that we will later
apply in a more general setting. As a byproduct, we will prove a simple, but surprising result in Lemma \ref{132}.

 Skew blocks of 132-avoiding permutations have a simple property that we state and prove below. Let $\mathcal Av n(q)$
denote the set of all permutations of length $n$ that avoid the pattern $q$. 

\begin{proposition} \label{structure} Let $p\in \mathcal Av _n(132)$ be skew indecomposable. Then $p$ ends in its
largest entry $n$. 
\end{proposition}

\begin{proof} Let us assume that $n$ is not in the last position of $p$. Then $p$ is not skew indecomposable, 
since every entry weakly on the left of $n$ must be larger than every entry strictly on the right of $n$, or a 132-pattern would be
obtained with $n$ playing the role of 3 in that pattern. 
\end{proof}

 Let $\mathcal Av _{n,\ell}(q)$
denote the set of all permutations of length $n$ that avoid the pattern $q$ and have $\ell$ skew blocks. 

Next we show the interesting fact that when $q=132$, then in (\ref{compare}), equality holds if $n>1$. 
\begin{lemma} \label{132} Let $n\geq 2$. Then the equality 
\[\textup{Av}_{n,2}(132) = \textup{Av}_{n,1}(132)\] holds.
\end{lemma}

\begin{proof}
We define a map $f:\mathcal Av_{n,2}(132) \rightarrow \mathcal Av_{n,1}(132)$, and show that it is a bijection.  Proposition \ref{structure} shows   that a 132-avoiding permutation is skew-indecomposable if and only if it ends in its maximum entry, the "if" part being obvious. 

Let $p\in \mathcal Av_{n,2}(132)$, and let us define $f(p)$ by moving the maximum entry $n$ of $p$ into the last position 
of $p$. It follows from the characterization of $\mathcal Av_{n,1}(132)$ given above that
 $f(p)\in \mathcal Av_{n,1}(132)$. 

If $w=w_1w_2\cdots w_{n-1}n$, then we obtain $f^{-1}(w)$ by moving its last entry to the immediate left of the rightmost
skew block $R$ of $w$.  This always results in a 132-avoiding permutation, since we placed $n$ between two skew blocks, and the obtained permutation will always have two skew blocks, namely $R$
and the rest of $f^{-1}(w)$, ending in $n$. So $f$ has an inverse function, and hence it is a bijection.
\end{proof}

\begin{example} If $p=534612$, then $f(p)=534126$.  \end{example}

\begin{theorem} \label{theo132} For all positive integers $n$, and all positive integers $\ell\leq n-1$, the inequality
\[\textup{Av}_{n,\ell+1}(132) \leq \textup{Av}_{n,\ell}(132)\] holds.
\end{theorem} 

\begin{proof} Applying equalities \eqref{powerrule} and \eqref{powerrule2} with $q=132$, we get
\[ A_{\ell,q}(z)= A_{\ell-1,q}(z) \cdot A_{1,q}(z), \] and
\[  A_{\ell+1,q}(z)= A_{\ell-1,q}(z) \cdot A_{2,q}(z).\]

The proof of our claim is now immediate, since Lemma \ref{132} shows that for each $n\geq 2$, the coefficient
of $z^n$ is at least as large in  $A_{1,q}(z)$ as in $ A_{2,q}(z)$. That is also true for $n=1$, since 
$1=\textup{Av}_{1,1}(132) > \textup{Av}_{1,2}(132)=0$. As all power series in the two equalities above have
nonnegative coefficients, this proves the statement of the theorem. 
\end{proof}

Note that the fact that $\textup{Av}_{1,1}(132) > \textup{Av}_{1,2}(132)$ implies that the inequality in Theorem 
\ref{theo132} is strict for all $n>1$ and $1<\ell\leq n-1$. 

\section{The case containing most patterns}

In the last section, we discussed a map that took a permutation with two skew blocks and moved its largest entry in 
its last position. For 132-avoiding permutations, this led to a bijection between two sets in which we were interested.
In this section, we will replace 132 by a pattern $q$ coming from a very large set of patterns. Furthermore, instead of moving the largest entry to the back, we will move the {\em last entry of the first skew block} into the end of the whole 
permutation. (In the special
case of $q=132$, that entry happens to be the largest entry as well.) We will be able to show that this map is an
{\em injection} from $\mathcal Av_{n,2}(q)$ to  $\mathcal Av_{n,1}(q)$.

{\em For the rest of this section, the pattern $q$ is assumed to be skew indecomposable.}
Let us call a pattern $q=q_1q_2\cdots q_k$  {\em good} if there does not exist a positive integer $i\leq k-1$ so that  $\{q_{k-i},q_{k-i+1},\cdots ,q_{k-1}\}=\{1,2,\cdots ,i\}$. That is, $q$ is good if there is no proper segment immediately 
preceding its last entry whose entries would be the smallest entries of $q$. For instance, $q=132$ and $q=3142$ are good,
but $q=1324$ and $q=35124$ are not, because of the choices of $i=3$ in the former, and $i=2$ in the latter. In particular,
$q$ is never good if $q_k=k$, because then we can choose $i=k-1$. 

\begin{lemma} \label{good}
Let $q$ be a good pattern. Then for all positive integers $n$, the inequality 
\[\textup{Av}_{n,2}(q) \leq  \textup{Av}_{n,1}(q)\] holds.\end{lemma}

\begin{proof}
We define a map $g:\mathcal Av_{n,2}(q) \rightarrow \mathcal Av_{n,1}(q)$, and show that it is an injection.

Let $p\in \mathcal Av_{n,2}(q)$. That means $p$ has two skew blocks; let us call the entries of the first skew block 
the {\em big} entries, and the entries of the second skew block the {\em small} entries.  Let us define $g(p)$ by moving the 
{\em rightmost big entry} $x$  of $p$ into the last position 
of $p$. The obtained permutation $g(p)$ still avoids $q$. Indeed, as $p$ avoids $q$, the only way $g(p)$ could possibly 
contain a copy $C$ of $q$ would be if $C$ contained the recently moved entry $x$ that is at the end of $g(p)$. However, $C$ could not consist entirely of big entries, since then $p$ would contain $q$ as well. 
Therefore, $C$ must start with a (possibly empty) string of big entries, followed by a non-empty string of small entries,
and end by its maximal entry $x$, which is a large entry. This contradicts our assumption that $q$ is a good pattern.

It is easy to see that $g(p)\in  \mathcal Av_{n,1}(q)$. Indeed, any cuts of $g(p)$ would necessarily cut the subsequence
of big entries (in $g(p)$, but also in $p$, since that subsequence does not change under the action of $g$) into skew blocks, and that would contradict our assumption that the big entries in $p$ form {\em one} skew block.

Now we prove that $g:\mathcal Av_{n,2}(q) \rightarrow \mathcal Av_{n,1}(q)$ is an injection. 
For a permutation $w=w_1w_2\cdots w_n$, let us define $h(w)$ as the result of moving $w_n$ immediately to the left
of the rightmost skew block of $w_1w_2\cdots w_{n-1}$. 

We claim that $h(g(p))=p$, for all permutations $p$ that are of length $n$ and have two skew blocks. Indeed, 
for such $p$, the image $g(p)$ is obtained the rightmost big entry $x$, that is, the entry immediately on the left of the 
last skew block, to the end of $p$. Setting $w=g(p)$, we have $w_n=x$, and $h$ moves $x$ back to its original position,
immediately to the left of the skew block of small entries of $p$. Indeed, the skew block of small entries of $p$ also forms
the rightmost skew block of $w_1w_2\cdots w_{n-1}$.

So we have seen that if $w\in \mathcal Av_{n,1}(q)$, then $w$ has at most one preimage under $g$, proving that 
 $g:\mathcal Av_{n,2}(q) \rightarrow \mathcal Av_{n,1}(q)$ is an injection, and hence proving our lemma. 
\end{proof}

Now we are going to extend the reach of Lemma \ref{good} to other patterns. 

\begin{lemma} \label{most}
Let $q=q_1\cdots q_k$ be a skew indecomposable pattern so that $q_1\neq 1$ or $q_k\neq k$ or both. 
Then the inequality \[\textup{Av}_{n,2}(q) \leq  \textup{Av}_{n,1}(q)\] holds.
\end{lemma}

\begin{proof}
Let $q=q_1q_2\cdots q_k$ be a pattern  that is not a good pattern and does not end in its largest entry. That means that there exists an $i<k-1$ so that $\{q_{k-i},q_{k-i+1},\cdots ,q_{k-1}\}=\{1,2,\cdots ,i\}$, and $q_k=y\neq k$. Therefore,
in the reverse $q^r$ of $q$, the entry $y\neq k$ is in the first position, and  the entries in positions $2,3,\cdots ,i+1$ are the entries $1,2,\cdots ,i$ in some order. In particular, the entry 1 precedes the entry $k$, so $q^r$ is skew indecomposable. Furthermore, $q^r$ is a good pattern, since again, the entry 1 precedes the entry $k$, so all ending segments that contain 1 also contain $k$,  so the only way for $q^r$ to be not good would be by ending in $k$. However, that would imply that $q$ starts in $k$, contradicting the assumption that $q$ is skew indecomposable. 

If $q$ is a skew indecomposable pattern that is not good and ends in its largest entry, but does not start in the entry 1, 
then the reverse  complement $(q^{c})^{rev}:=q^{rc}$ of $q$ is a skew indecomposable pattern that does not end in its largest entry. So, by the previous paragraph, either $q^{rc}$ or its reverse $q^{c}$ is a good pattern. In either case, we finish our proof 
by applying Lemma \ref{good} to either $q^{rc}$ or to $q^c$, and then applying Proposition \ref{equality} to conclude that our statement holds for $q$ as well. 
\end{proof}

Lemma \ref{most} does not cover patterns that start with their minimal element and end with their largest element,
like 1324. However, if $q$ is such a pattern, we can still prove the statement of Lemma \ref{most} for $q$ if $q$ is
Wilf-equivalent to a pattern $q'$ that {\em is} covered by Lemma \ref{most}. Indeed, this is an immediate consequence
of Proposition \ref{equality}. So, for instance, the statement of Lemma \ref{most} also holds for all monotone
patterns $12\cdots k$, since it is well-known \cite{babson} that $12\cdots k$ is Wilf-equivalent to the pattern $12\cdots (k-2)k(k-1)$. 

The proof of the monotonicity result announced in the introduction is now immediate.
\begin{theorem} \label{mongen}
Let $q=q_1\cdots q_k$ be a skew indecomposable pattern so that at least one of the following conditions hold
\begin{enumerate}
\item $q_1\neq 1$, or
\item  $q_k\neq k$, or
\item $q_1=1$ and $q_k=k$, but $q$ is Wilf-equivalent to a skew-indecomposable pattern in which the first entry is
not 1 or the last entry is not $k$. 
\end{enumerate} 
Then the inequality \[\textup{Av}_{n,\ell+1}(q) \leq  \textup{Av}_{n,\ell}(q)\] holds for all nonnegative integers $n$ and all positive integers $\ell$.
\end{theorem}

\begin{proof} Proposition \ref{equality} implies that we can assume that $q$ does not start with 1, or does not end in $k$. Then the proof of our claim  is immediate from Lemma \ref{most} and Proposition \ref{compare}. 
\end{proof}

\section{Why $A_q(z)$ is not rational} \label{secnonrat}
We can now prove the result mentioned in the title of the paper. 
\begin{theorem} \label{nonrattheo}
Let  $q=q_1q_2\cdots q_k$ be a pattern so that either $\{1,k\}\neq \{q_1,q_k\}$, or $q$ is 
Wilf-equivalent to a pattern $v=v_1v_2\cdots v_k$ so that $\{1,k\}\neq \{v_1,v_k\}$
Then the generating function $A_q(z)$ is not rational. 
\end{theorem}

\begin{proof}  First, note that we can assume that $q$ is skew indecomposable. Indeed, if $q$ is not, then $q^r$ is, 
and clearly, $A_q(z)=A_{q^r}(z)$. 

So let $q$ be skew indecomposable, and let us assume that $A_q(z)$ is rational. Then by \eqref{indec}, the power series $A_{1,q}(z)$ is
also rational.   Let $R>0$ be the radius of convergence of $A_{1,q(z)}$. We know that $R>0$, since we know \cite{marcus} that
$\textup{Av}_{n,1}(q)\leq \textup{Av}_n(q) \leq c_q^n$ for some constant $c_q$.  As the coefficients of $A_{1,q}(z)$ are all nonnegative real numbers,  it follows
from Pringsheim's theorem (Theorem IV.6 in \cite{flajolet}) that the positive real number $R$ is a singularity of 
$A_{1,q}(z)$. As $A_{1,q}(z)$ is rational, $R$ is a pole of $A_{1,q}(z)$, so $\lim_{z\rightarrow r} A_{1,q}(z) =\infty$. Therefore, there exists a positive real number
$z_0<R$ so that $A_{1,q}(z_0)>1$. Therefore, 
\[\sum_{n \geq 1} \textup{Av}_{n,1}(q) z_0^n =A_{1,q}(z_0) < A_{1,q}(z_0)^2=A_{2,q}(z_0)=\sum_{n \geq 2} \textup{Av}_{n,2}(q) z_0^n, \]
contradicting the fact, proved in Theorem \ref{mongen},  that for each $n$, the coefficient of $z^n$ in the  leftmost powers series is at least as large as it is in the rightmost power series.
\end{proof}

The elegant argument in the previous paragraph is due to Robin Pemantle \cite{pemantle}. It shows that the square of 
a rational power series with nonnegative coefficients and a positive radius of convergence  will have at least one coefficient
that is larger than the corresponding coefficient of the power series itself. A significantly more complicated argument
proves a stronger statement. The interested reader should consult \cite{bell} for details. 

\section{Broader context: Supercritical relations} \label{broader}
We will place our results into the broader context of {\em supercritical relations}. Readers who are interested to learn
more about this subject are invited to consult Sections V.2 and VI.9 of \cite{flajolet}. 

\begin{definition} \label{supcrit}  Let $F$ and $G$ be two generating functions with nonnegative real coefficients that are analytic at 0,
and let us assume that $G(0)=0$. Then the relation
\[F(z)=\frac{1}{1-G(z)} \]
is called {\em supercritical} if $G(R_G)>1$, where $R_G$ is the radius of convergence of $G$. 
\end{definition}

Note that as the coefficients of $G(z)$ are nonnegative, $G(R_G)>1$ implies that $G(\alpha)=1$ for some $\alpha \in
(0,R_G)$. So, if the relation between $F$ and $G$ described above is supercritical, then the radius of convergence of
$F$ is less than that of $G$, and so the exponential growth rate of the coefficients of $F$ is larger than that of 
$G$. 

\begin{theorem} \label{notsup} Let $q$ be any permutation pattern satisfying the conditions of Theorem \ref{mongen}.
Then the relation 
\[A_q(z) = \frac{1}{1-A_{1,q}(z)} \] is not supercritical.
\end{theorem}  

\begin{proof} It is immediate from Theorem \ref{mongen} that we have
\[\textup{Av}_n(q) =\sum_{\ell=1}^n \textup{Av}_{n,\ell}(q)=n\textup{Av}_{n,1}(q),\]
implying that the sequences  $\textup{Av}_n(q)$ and $\textup{Av}_{n,1}(q)$ have the same exponential order. 
By Theorem \ref{notsup}, that means that the relation between their generating functions cannot be supercritical.
\end{proof}

On the other hand, combinatorial generating functions that are rational lead to supercritical 
relations, as the following extension of Theorem \ref{nonrattheo} shows.

\begin{theorem} \label{general}  Let $G(z)$ be a rational power series with nonnegative real coefficients
that satisfies $G(0)=0$. Then the relation 
\[F(z)= \frac{1}{1-G(z)} \] is supercritical.
\end{theorem}

\begin{proof} If $G(z)$ is a polynomial, then $R_G=\infty$, so $G(R_G)=\infty>1$, and our claim is proved.  
Otherwise,  $G(z)$ is a rational function that has at least one singularity, and all its singularities are poles.
Let $R_G$ be a singularity of smallest modulus. Then $G(R_G)=\infty>1$, completing our proof. 
\end{proof} 

Now we see that Theorem \ref{nonrattheo} immediately follows from the two results in this section. Indeed, if 
$q$ is a pattern satisfying the conditions of Theorem \ref{mongen}, then $A_q(z)$ cannot be rational, because if it
was, then so would be $A_{1,q}(z)$. Therefore, by Theorem \ref{general}, the relation $A_q(z) = \frac{1}{1-A_{1,q}(z)}$
would be supercritial, but we know by Theorem \ref{notsup} that it is not.

\section{Further directions}
It goes without saying that it is an intriguing problem to prove Lemma \ref{most} for the remaining patterns.
Of course, Theorem \ref{nonrattheo} could possibly be proved by other means, but numerical evidence seems to suggest
that Theorem \ref{nonrattheo} will hold even for patterns that start with their minimum entry and end in their largest 
entry. Interestingly, the shortest patterns for which we cannot prove Theorem \ref{nonrattheo} are 1324 and 4231,
which also happen to be the shortest patterns for which no exact formula is known for $\textup{Av}_n(q)$. 

It is important to point out that our results do not hold at all for permutation classes that are generated by more than
one pattern. For instance, let $\textup{Av}_n(123,132)$ denote the number of permutations of length $n$ that avoid
both 123 and 132. It is then easy to prove that $\textup{Av}_n(123,132)=2^{n-1}$, so $ A_{123,132}(z)=(1-z)/(1-2z)$, a rational function. 
Note that in this case, $\textup{Av}_{n,1}(123,132)=1$, since the only such permutation is $(n-1)(n-2)\cdots 1n$, 
while $\textup{Av}_{n,2}(123,132)=n-1$, so Lemma \ref{most} does not hold. 

\begin{center}  {\bf Acknowledgment}  \end{center}

I am grateful to Robin Pemantle, Steve Melczer, Marni Mishna, Bruno Salvy, Vincent Vatter, and Zachary Hamaker for helpful conversations and advice, and to Michael Cory for computing help. I am also indebted to the two anonymous referees whose careful reading and expert advise improved this paper.  My research is partially supported by Simons Collaboration Grant  421967.


\begin{thebibliography}{99}
\bibitem{babson} E. Babson, J. West, The permutations $123p_4\cdots p_m$  and $321p_4\cdots p_m$ are Wilf-equivalent. {\em Graphs Combin}. {\bf 16} (2000), no. 4, 373--380.
\bibitem{bell} J. P. Bell, S. Gerhold, On the positivity set of a linear recurrence sequence, {\em Israel J. Math} {\bf 157} (2007), 333--345.
\bibitem{awalk} M. B\'ona, A Walk Through Combinatorics, 4th edition, World Scientific, 2016.
\bibitem{combperm}  M. B\'ona, Combinatorics of Permutations, 2nd edition, CRC Press, 2012.
\bibitem{survey} M. B\'ona, A Survey of Stack Sortable Permutations. In: {\em 50 Years of Combinatorics, Graph Theory, 
and Computing}. CRC Press, Boca Raton FL, to appear. 
\bibitem{flajolet} P. Flajolet, R. Sedgewick, Analytic Combinatorics, Cambridge University Press, 2009. 
\bibitem{garrabrant} S. Garrabrant, I. Pak, Pattern avoidance is not P-recursive. Preprint, available at 
\url{https://arxiv.org/pdf/1505.06508.pdf}.
\bibitem{garrabrant1} S. Garrabrant, I. Pak, Permutation patterns are hard to count. {\em  
Proceedings of the Twenty-Seventh Annual ACM-SIAM Symposium on Discrete Algorithms.} 
\url{https://doi.org/10.1137/1.9781611974331.ch66}.
\bibitem{marcus} A. Marcus, G. Tardos, Excluded permutation matrices and the Stanley-Wilf conjecture. {\em J. Combin. Theory Ser. A} {\bf 107} (2004), no. 1, 153--160.
\bibitem{noonan} J. Noonan, D. Zeilberger, The enumeration of permutations with a prescribed number of "forbidden'' patterns. {\em Adv. in Appl. Math.} {\bf  17} (1996), no. 4, 381--407.
\bibitem{oeis} Online Encyclopedia of Integer Sequences, {\em online database}, \url{http://oeis.org/}.
\bibitem{pemantle} R. Pemantle, Personal communication, January 17, 2019.
\bibitem{stanleyec2} R. Stanley, Enumerative Combinatorics, Volume II, Cambridge University Press, 1997.
\bibitem{vatter} V. Vatter, Permutation classes. In: Handbook of Enumerative Combinatorics, Mikl\'os B\'ona, editor,
CRC Press, 2015.
\end{thebibliography}
\end{document}